\documentclass[a4, 12pt]{amsart}
\usepackage[mathscr]{eucal}
\usepackage{amssymb}
\usepackage{latexsym}
\usepackage{amsthm}
\usepackage{color}
\theoremstyle{plain}
\newtheorem{theorem}{Theorem}[section]

\newtheorem{remark}{Remark}[section]
\newtheorem{lemma}{Lemma}[section]
\newtheorem{proposition}{Proposition}[section]

\makeatletter
\@addtoreset{equation}{section}

\title[The rigidity theorems for complete self-expander]
{The rigidity theorems for complete self-expander of mean curvature flow}
\author [Z. Li and G. Wei]{Zhi Li and Guoxin Wei}
\address{Zhi Li \\  School of Mathematics and Statistics, Henan Normal University,
\newline \indent 453007, Xinxiang, Henan, China. \newline \indent lizhihnsd@126.com}

\address{Guoxin Wei \\  School of Mathematical Sciences, South China Normal University,
\newline \indent 510631, Guangzhou,  China. \newline \indent  weiguoxin@tsinghua.org.cn}

\begin{document}
\maketitle

\begin{abstract}
Our first main result is a rigidity theorem for complete self-expanders in the Euclidean space with higher codimension, assuming an integral curvature pinching condition. More precisely, any smooth complete self-expander $x:M\to \mathbb R^{n+p}$($n\geq3$) that satisfies both
$(\int_{M}|A|^{n}d\mu)^{\frac{1}{n}}<K(n)$ and $\int_{M}|A|^{n}e^{\frac{|x|^{2}}{2}}d\mu<\infty$ for a positive constant $K(n)$ depending only on the dimension $n$ must be isometric to $\mathbb R^{n}$.
Moreover, we show that the rigidity result persists when the pinching condition is expressed in terms of the trace-free second fundamental form.
\end{abstract}



\section{introduction}
\vskip2mm
\noindent

An $n$-dimensional  submanifold  $x: M\rightarrow \mathbb{R}^{n+p}$  in the $(n+p)$-dimensional
Euclidean space $\mathbb{R}^{n+p}$  is called a self-expander if it satisfies
\begin{equation*}
\vec H=x^{\perp},
\end{equation*}
where $\vec H$ and $x^{\perp}$ denote mean curvature vector of  this submanifold and the normal part of the position vector $x$, respectively.
Self-expanders are self similar solutions of the mean curvature flow,
that is, the family of hypersurfaces $x_{t}=\sqrt{2t}x, \ \ t>0$ satisfying the equation of mean curvature flow.
Self-expanders are important as they model the behavior of a mean curvature flow
coming out of a conical singularity (\cite{AIC}), and also model the long time behaviors of the flows starting from entire graphs (\cite{Guo}).

Self-expanders arise naturally when one considers solutions of graphical mean curvature
flow. In the case of codimension one and under certain assumptions on the initial hypersurface
at infinity, Ecker and Huisken (\cite{EH}) studied the mean curvature flow evolutions of entire graphical immersions. Under some assumptions
on the initial hypersurface at infinity, they showed that the solution of the mean curvature flow exists for all times $t>0$
and converges to a self-expander. Later, Stavrou (\cite{Sta}) proved the same result under weaker hypothesis
that the initial hypersurface has a unique tangent cone at infinity. Self-expanders also appears in the
mean curvature evolution of cones. In \cite{Ilm}, Ilmanen studied the existence of E-minimizing self-expanding hypersurfaces which converge to prescribed closed cones at infinity.  Ding (\cite{Din}) studied self-expanders and their relationship to minimal cones in Euclidean space. The space
of asymptotically conical self-expanders was studied in several papers by Bernstein and
Wang (\cite{BW01} and \cite{BW02}). There are other works in
self-expanders (see, e.g. (\cite{BW03}, \cite{FM}, \cite{LN}, \cite{WX}) and references therein).

As with self-shrinkers in the  mean curvature flow, self-expanders also admit a wealth of classification results and rigidity theorems.
Ishimura \cite{Ish} and Halldorsson \cite{H} have classified self-expander curves in $\mathbb R^{2}$. In \cite{H}, Halldorsson states that each of the complete self-expander curves immersed in $\mathbb R^{2}$ is convex, properly embedded and asymptotic to the boundary of a cone with vertex at the origin. In 2018, Cheng and Zhou \cite{CZ} studied some properties of complete properly immersed self-expanders. Later, Ancari and Cheng \cite{AC} studied immersed self-expander hypersurfaces whose mean curvatures have some linear growth controls and obtained some rigidity property of hyperplanes as self-expanders in the Euclidean space. In 2024, the author of the present paper (\cite{LW1}) studied the classification of complete the complete Lagrangian self-expanders in $\mathbb{C}^{2}$.
Recently, under the condition that the squared norm of the second fundamental form being bounded from above, the author of the present paper (\cite{LW2}) studied the rigidity problem for $2$-dimensional complete Lagrangian self-expanders with constant squared norm $|\vec{H}|^{2}$ of the mean curvature vector in $\mathbb{C}^{2}$.

In higher dimension and higher codimension case, the classification and rigidity of self-expander is much more complicated. In this paper, we first study a rigidity theorem for complete self-expanders under an integral curvature pinching condition.

\begin{theorem}\label{theorem 1.1}
Let $x: M^{n}\to \mathbb{R}^{n+p}$ ($n\geq3$) be a complete immersed self-expander in $\mathbb{R}^{n+p}$. If $M^{n}$ satisfies integral
conditions $$\Big(\int_{M^{n}}|A|^{n}\Big)^{\frac{1}{n}}<K(n) \ \ and \ \ \int_{M}|A|^{n}e^{\frac{|x|^{2}}{2}}d\mu<\infty,$$
where $K(n)$ is an explicit positive constant depending only on $n$, then $M^{n}$ is isometric to $\mathbb R^{n}$.
\end{theorem}

\begin{remark}
In fact, this theorem is mainly influenced by the earlier work of Professors Ding and Xin (\cite{DX}) on self-shrinkers.
\end{remark}
Consider a submanifold with trace-free second fundamental form $\mathring{A}=A-\frac{1}{n}g\otimes \vec{H}$.
The another main goal of this paper is to prove a rigidity theorem for self-expanders, assuming integral curvature pinching conditions involving $\mathring{A}$.

\begin{theorem}\label{theorem 1.2}
Let $x: M^{n}\to \mathbb{R}^{n+p}$ ($n\geq3$) be a complete immersed self-expander in $\mathbb{R}^{n+p}$. If
the trace-free second fundamental form $\mathring{A}$ of $M^{n}$ satisfies $$\Big(\int_{M^{n}}|\mathring{A}|^{n}\Big)^{\frac{1}{n}}<\Omega^{-\frac{1}{2}}(a, n) \ \ and \ \ \int_{M}|\mathring{A}|^{2a}e^{\frac{|x|^{2}}{2}}d\mu<\infty,$$
where
$$1\leq a<\frac{n+\sqrt{n^{2}-2n}}{2}, \ \ \Omega(a, n)=\frac{4a^{2}\varsigma D^{2}(n)[\frac{(2a-1)(n-2)^{2}}{2n(2an-n-2a^{2})}+(n-1)^{2}]}{(2a-1)(n-2)^{2}},$$
$\varsigma=1$ if $p=1$, $\varsigma=2$ if $p\geq2$
and $D(n)$ is the Sobolev constant, then $M^{n}$ is isometric to $\mathbb R^{n}$.
\end{theorem}

It is worth noting that when $a=\frac{n}{2}$, we obtain the following rigidity property, whose hypotheses are weaker than those of Theorem \ref{theorem 1.1}.

\begin{theorem}\label{theorem 1.3}
Let $x: M^{n}\to \mathbb{R}^{n+p}$ ($n\geq3$) be a complete immersed self-expander in $\mathbb{R}^{n+p}$. If
the trace-free second fundamental form $\mathring{A}$ of $M^{n}$ satisfies $$\Big(\int_{M^{n}}|\mathring{A}|^{n}\Big)^{\frac{1}{n}}<\Omega^{-\frac{1}{2}}(n) \ \ and \ \ \int_{M}|\mathring{A}|^{n}e^{\frac{|x|^{2}}{2}}d\mu<\infty,$$
where
$$\Omega(n)=\frac{\varsigma D^{2}(n)[n^{2}(n-1)+n-2]}{(n-2)^{2}},$$
$\varsigma=1$ if $p=1$, $\varsigma=2$ if $p\geq2$
and $D(n)$ is the Sobolev constant, then $M^{n}$ is isometric to $\mathbb R^{n}$.
\end{theorem}

\vskip5mm
\section {Preliminaries}
\vskip2mm

\noindent

Let $x: M^{n} \rightarrow\mathbb{R}^{n+p}$ denote a smooth complete isometric immersion of an $n$-dimensional manifold into an $(n+p)$-dimensional Euclidean space
$(\mathbb{R}^{n+p}, \bar g)$ with the induced metric $g$.
In this paper, unless otherwise specified, the notations with a bar, for instance $\bar\nabla$ and $\bar\Delta$ denote the quantities corresponding the metric $\bar g$ on $\mathbb{R}^{n+p}$. On the other
hand, the notations like $\nabla$ and $\Delta$ denote the quantities corresponding the intrinsic metric $g$ on $M^{n}$.
Let $A$ denote the second fundamental form of $(M^{n}, g)$, that is, at $p \in M^{n}$, $A(X, Y)=(\bar\nabla_{X} Y )^{N}$, where $X, Y \in T_{p}M^{n}$ and $(\cdot)^{N}$ denotes the projection onto the normal bundle of $M^{n}$. Similarly, $(\cdot)^{T}$ stands for the tangential projection. $\vec{H}$ denote the mean curvature vector of $M^{n}$ given by $\vec{H}=trace (A)$.

We take a local orthonormal frame field $\{e_{i}, e_{\alpha}\}$ along $M^{n}$, together with its dual coframe
$\{\omega_{i}, \omega_{\alpha}\}$, where the $e_{i}$ are tangent to $M^{n}$ and the $e_{\alpha}$ are normal to $M^{n}$.
Denote by $g$ the induced metric on $M^{n}$ and the induced structure equations of $M^{n}$ are

\begin{equation*}
d\omega_{i}=\sum_{j} \omega_{ij}\wedge\omega_{j}, \quad  \omega_{ij}=-\omega_{ji},
\end{equation*}
\begin{equation*}
d\omega_{ij}=\sum_{k} \omega_{ik}\wedge\omega_{kj}+\sum_{\alpha} \omega_{i\alpha}\wedge\omega_{\alpha j},
\end{equation*}
\begin{equation*}
\Omega_{ij}=d\omega_{ij}-\sum_{k} \omega_{ik}\wedge\omega_{kj}=-\frac{1}{2}\sum_{k,l}
R_{ijkl} \omega_{k}\wedge\omega_{l}.
\end{equation*}
Here and in the sequel we use the following conventions on the ranges of indices:
$$1\leq i,j,k\leq\cdots\leq n; \ \ n+1\leq \alpha,\beta,\gamma\leq\cdots\leq n+p.$$

\noindent By restricting  these forms to $M^{n}$, we get
\begin{equation*}
\omega_{\alpha i}=\sum_{j}h^{\alpha}_{ij}\omega_{j}, \ \ h^{\alpha}_{ij}=h^{\alpha}_{ji},
\end{equation*}

\begin{equation*}
A=\sum_{i,j,\alpha}h^{\alpha}_{ij}\omega_{i}\otimes\omega_{j}\otimes\omega_{\alpha}, \ \ \vec{H}=\sum_{\alpha}H^{\alpha}e_{\alpha}=\sum_{\alpha,i}h^{\alpha}_{ii}e_{\alpha},
\end{equation*}

\begin{equation*}
R_{ijkl}=\sum_{\alpha}(h^{\alpha}_{ik}h^{\alpha}_{jl}-h^{\alpha}_{il}h^{\alpha}_{jk}),  \ \ R_{\alpha\beta kl}=\sum_{i}(h^{\alpha}_{ik}h^{\beta}_{il}-h^{\alpha}_{il}h^{\beta}_{ik}).
\end{equation*}

\noindent
Defining the first and second covariant derivatives of $h_{ij}$ by
\begin{equation*}
\sum_{k}h^{\alpha}_{ijk}\omega_k=dh^{\alpha}_{ij}+\sum_{k}h^{\alpha}_{ik}\omega_{kj}
+\sum_{k} h^{\alpha}_{kj}\omega_{ki}+\sum_{\beta} h^{\beta}_{ij}\omega_{\beta\alpha},
\end{equation*}

\begin{equation*}
\sum_{l}h^{\alpha}_{ijkl}\omega_{l}=dh^{\alpha}_{ijk}+\sum_{l}h^{\alpha}_{ljk}\omega_{li}
+\sum_{l}h^{\alpha}_{ilk}\omega_{lj}+\sum_{l} h^{\alpha}_{ijl}\omega_{lk}+\sum_{\beta} h^{\beta}_{ijk}\omega_{\beta\alpha},
\end{equation*}
we obtain
\begin{equation*}
h^{\alpha}_{ijk}=h^{\alpha}_{ikj},
\end{equation*}

\begin{equation*}
h^{\alpha}_{ijkl}-h^{\alpha}_{ijlk}=\sum_m
h^{\alpha}_{mj}R_{mikl}+\sum_m h^{\alpha}_{im}R_{mjkl}+\sum_m h^{\beta}_{ij}R_{\beta\alpha kl}.
\end{equation*}

For self-expanders, an elliptic operator $\mathcal{L}$ is defined by
\begin{equation*}
\mathcal{L}(\cdot)=\Delta (\cdot)+\langle x,\nabla (\cdot)\rangle=\rho^{-1} div(\rho\nabla (\cdot)), \ \ \rho=e^{\frac{|x|^{2}}{2}},
\end{equation*}
where $\Delta$, $div$ and $\nabla$ denote the Laplacian, the divergent operator  and the gradient operator, respectively.

Direct computations using above formulas and the Ricci identities easily give the following Simons type identities.
\begin{equation}\label{2.1-1}
\frac{1}{2}\mathcal{L}|\vec{H}|^{2}=|\nabla^{N}\vec{H}|^{2}-|\vec{H}|^{2}-\sum_{i,j}(\sum_{\alpha}H^{\alpha}h^{\alpha}_{ij})^{2},
\end{equation}

\begin{equation}\label{2.1-2}
\aligned
\frac{1}{2}\mathcal{L}|A|^{2}
=&|\nabla A|^{2}-|A|^{2}-\sum_{\alpha,\beta}(\sum_{i,j}h^{\alpha}_{ij}h^{\beta}_{ij})^{2} \\
&-\sum_{i,j,\alpha,\beta}\Big(\sum_{k}(h^{\alpha}_{ik}h^{\beta}_{kj}-h^{\alpha}_{jk}h^{\beta}_{ki})\Big)^{2}.
\endaligned
\end{equation}
Next, we define the trace-free second fundamental form by $\mathring{A}=A-\frac{1}{n}g\otimes \vec{H}$, then
\begin{equation*}
|\mathring{A}|^{2}=|A|^{2}-\frac{1}{n}|\vec{H}|^{2}, \ \ |\nabla\mathring{A}|^{2}=|\nabla A|^{2}-\frac{1}{n}|\nabla^{N}\vec{H}|^{2}.
\end{equation*}
Subsequently, through \eqref{2.1-1} and \eqref{2.1-2}, we obtain
\begin{equation*}
\aligned
\frac{1}{2}\mathcal{L}|\mathring{A}|^{2}
=&|\nabla\mathring{A}|^{2}-|\mathring{A}|^{2}+\frac{1}{n}\sum_{i,j}(\sum_{\alpha}H^{\alpha}h^{\alpha}_{ij})^{2}-\sum_{\alpha,\beta}(\sum_{i,j}h^{\alpha}_{ij}h^{\beta}_{ij})^{2} \\
&-\sum_{i,j,\alpha,\beta}\Big(\sum_{k}(h^{\alpha}_{ik}h^{\beta}_{kj}-h^{\alpha}_{jk}h^{\beta}_{ki})\Big)^{2}.
\endaligned
\end{equation*}
In particular, for the codimension one case ($p=1$), the above Simons type identities simplifies to the following expression.
\begin{equation*}
\frac{1}{2}\mathcal{L}|A|^{2}=|\nabla A|^{2}-|A|^{2}-|A|^{4}.
\end{equation*}
\begin{equation*}
\aligned
\frac{1}{2}\mathcal{L}|\mathring{A}|^{2}
=&|\nabla\mathring{A}|^{2}-|\mathring{A}|^{2}-|\mathring{A}|^{4}-\frac{1}{n}|\vec{H}|^{2}|\mathring{A}|^{2}.
\endaligned
\end{equation*}
For high codimension ($p\geq2$), there are the following estimates (see \cite{CX},\cite{LL},\cite{WXZ})
\begin{equation*}
\sum_{\alpha,\beta}(\sum_{i,j}h^{\alpha}_{ij}h^{\beta}_{ij})^{2} \\
+\sum_{i,j,\alpha,\beta}\Big(\sum_{k}(h^{\alpha}_{ik}h^{\beta}_{kj}-h^{\alpha}_{jk}h^{\beta}_{ki})\Big)^{2}\leq\frac{3}{2}|A|^{4},
\end{equation*}

\begin{equation*}
\aligned
&\frac{1}{n}\sum_{i,j}(\sum_{\alpha}H^{\alpha}h^{\alpha}_{ij})^{2}-\sum_{\alpha,\beta}(\sum_{i,j}h^{\alpha}_{ij}h^{\beta}_{ij})^{2}
-\sum_{i,j,\alpha,\beta}\Big(\sum_{k}(h^{\alpha}_{ik}h^{\beta}_{kj}-h^{\alpha}_{jk}h^{\beta}_{ki})\Big)^{2}\\
\geq&-2|\mathring{A}|^{4}-\frac{1}{n}|\vec{H}|^{2}|\mathring{A}|^{2}.
\endaligned
\end{equation*}

We summarize the above calculations as the following lemma.
\begin{lemma}\label{lemma 2.1}
The squared norm of the second fundamental form $|A|^{2}$ and the trace-free second fundamental form $|\mathring{A}|^{2}$ of the self-expanders in $\mathbb{R}^{n+p}$ satisfies following Simons type inequalities:
\begin{equation}\label{2.1-3}
\frac{1}{2}\mathcal{L}|A|^{2}\geq|\nabla |A||^{2}-|A|^{2}-\tau|A|^{4},
\end{equation}

\begin{equation}\label{2.1-4}
\aligned
\frac{1}{2}\mathcal{L}|\mathring{A}|^{2}
\geq&|\nabla|\mathring{A}||^{2}-|\mathring{A}|^{2}-\varsigma|\mathring{A}|^{4}-\frac{1}{n}|\vec{H}|^{2}|\mathring{A}|^{2},
\endaligned
\end{equation}
where we have used the Kato-type inequalities $|\nabla A|^{2}\geq|\nabla |A||^{2}$ and $|\nabla\mathring{A}|^{2}\geq|\nabla|\mathring{A}||^{2}$; $\tau=1$ and $\varsigma=1$ if $p=1$, $\tau=\frac{3}{2}$ and $\varsigma=2$ if $p\geq2$. \end{lemma}

 \vskip10mm
\section{Proof of the rigidity theorem for $A$}

\vskip2mm
\noindent

There is the Sobolev inequality (see \cite{MS}) as follows

\begin{equation}\label{3.1-1}
\kappa^{-1}\Big(\int_{M^{n}}f^{\frac{2n}{n-2}}d\mu\Big)^{\frac{n-2}{n}}\leq\int_{M^{n}}|\nabla f|^{2}d\mu+\frac{1}{2}\int_{M^{n}}|\vec{H}|^{2}|f|^{2}d\mu, \ \ \forall f\in \mathrm{C}^{\infty}_{c}(M^{n}),
\end{equation}
where $\kappa>0$ is a constant.

\noindent
Let $\eta$ be a smooth function with compact support in $M^{n}$, and define $\rho=e^{\frac{|x|^{2}}{2}}$, we obtain
\begin{equation}\label{3.1-2}
\nabla_{e_{i}} \rho=\rho \langle x, e_{i}\rangle, \ \ \nabla_{e_{i}} \rho^{\frac{1}{2}}=\frac{1}{2}\rho^{\frac{1}{2}}\langle x, e_{i}\rangle, \ \
\Delta\rho=\rho (n+|x|^{2}).
\end{equation}

{\it Proof of Theorem \ref{theorem 1.1}}.
Multiplying $|A|^{n-2}\eta^{2}\rho$ on both sides of the \eqref{2.1-3} and integrating by parts yield
\begin{equation*}
\aligned
0\geq&-\frac{1}{2}\int_{M^{n}}|A|^{n-2}\eta^{2}\rho\mathcal{L}|A|^{2}+\int_{M^{n}}|\nabla |A||^{2}|A|^{n-2}\eta^{2}\rho-
\int_{M^{n}}|A|^{n}\eta^{2}\rho \\
&-\tau\int_{M^{n}}|A|^{n+2}\eta^{2}\rho \\
=&\int_{M^{n}}\langle \nabla |A|, \nabla(|A|^{n-2}\eta^{2})\rangle|A|\rho+\int_{M^{n}}|\nabla |A||^{2}|A|^{n-2}\eta^{2}\rho-
\int_{M^{n}}|A|^{n}\eta^{2}\rho \\
&-\tau\int_{M^{n}}|A|^{n+2}\eta^{2}\rho \\
=&(n-1)\int_{M^{n}}|\nabla |A||^{2}|A|^{n-2}\eta^{2}\rho+2\int_{M^{n}}\langle \nabla |A|, \nabla\eta)\rangle|A|^{n-1}\eta\rho-
\int_{M^{n}}|A|^{n}\eta^{2}\rho \\
&-\tau\int_{M^{n}}|A|^{n+2}\eta^{2}\rho.
\endaligned
\end{equation*}
By the Cauchy inequality, for any $\varepsilon>0$, we have
\begin{equation}\label{3.1-3}
\aligned
(n-1-\varepsilon)\int_{M^{n}}|\nabla |A||^{2}|A|^{n-2}\eta^{2}\rho\leq&
\int_{M^{n}}|A|^{n}\eta^{2}\rho+\tau\int_{M^{n}}|A|^{n+2}\eta^{2}\rho \\
&+\frac{1}{\varepsilon}\int_{M^{n}}|\nabla\eta|^{2}|A|^{n}\rho.
\endaligned
\end{equation}
Setting $f=|A|^{\frac{n}{2}}\eta\rho^{\frac{1}{2}}$. Integrating by parts and using \eqref{3.1-2}, one obtain

\begin{equation}\label{3.1-4}
\aligned
\int_{M^{n}}|\nabla f|^{2}
=&\int_{M^{n}}|\nabla (|A|^{\frac{n}{2}}\eta)|^{2}\rho+\int_{M^{n}}|A|^{n}\eta^{2}|\nabla\rho^{\frac{1}{2}}|^{2}+\frac{1}{2}\int_{M^{n}}\langle \nabla \rho, \nabla(|A|^{n}\eta^{2})\rangle \\
=&\int_{M^{n}}|\nabla (|A|^{\frac{n}{2}}\eta)|^{2}\rho+\frac{1}{4}\int_{M^{n}}|A|^{n}\eta^{2}\rho|x^{T}|^{2}-\frac{1}{2}\int_{M^{n}}|A|^{n}\eta^{2}\Delta\rho \\
=&\int_{M^{n}}|\nabla (|A|^{\frac{n}{2}}\eta)|^{2}\rho-\frac{1}{4}\int_{M^{n}}|A|^{n}\eta^{2}\rho|x^{T}|^{2}-\frac{1}{2}\int_{M^{n}}|\vec{H}|^{2}|A|^{n}\eta^{2}\rho \\
&-\frac{n}{2}\int_{M^{n}}|A|^{n}\eta^{2}\rho.
\endaligned
\end{equation}
Substituting \eqref{3.1-3} and \eqref{3.1-4} into \eqref{3.1-1}, and using the Cauchy inequality, for any $\delta>0$, we know
\begin{equation}\label{3.1-5}
\aligned
&\kappa^{-1}\Big(\int_{M^{n}}f^{\frac{2n}{n-2}}d\mu\Big)^{\frac{n-2}{n}}\\
\leq&\int_{M^{n}}|\nabla (|A|^{\frac{n}{2}}\eta)|^{2}\rho-\frac{1}{4}\int_{M^{n}}|A|^{n}\eta^{2}\rho|x^{T}|^{2}-\frac{n}{2}\int_{M^{n}}|A|^{n}\eta^{2}\rho \\
\leq&\frac{n^{2}}{4}(1+\delta)\int_{M^{n}}|\nabla |A||^{2}|A|^{n-2}\eta^{2}\rho+(1+\frac{1}{\delta})\int_{M^{n}}|\nabla \eta|^{2}|A|^{n}\rho-\frac{n}{2}\int_{M^{n}}|A|^{n}\eta^{2}\rho \\
\leq&\frac{n^{2}(1+\delta)}{4(n-1-\varepsilon)}\Big(\int_{M^{n}}|A|^{n}\eta^{2}\rho+\tau\int_{M^{n}}|A|^{n+2}\eta^{2}\rho+\frac{1}{\varepsilon}\int_{M^{n}}|\nabla\eta|^{2}|A|^{n}\rho\Big)\\
&+(1+\frac{1}{\delta})\int_{M^{n}}|\nabla \eta|^{2}|A|^{n}\rho
-\frac{n}{2}\int_{M^{n}}|A|^{n}\eta^{2}\rho.
\endaligned
\end{equation}
Let $\delta=\delta(\varepsilon)=\frac{2(n-1-2\varepsilon)}{n}-1$, where $\varepsilon$ is a positive constant such that $\delta>0$. Then by the H\"{o}lder inequality,
\eqref{3.1-5} yields
\begin{equation}\label{3.1-6}
\aligned
&\kappa^{-1}\Big(\int_{M^{n}}f^{\frac{2n}{n-2}}d\mu\Big)^{\frac{n-2}{n}}\\
\leq&\frac{\tau n}{2}\cdot\frac{n-1-2\varepsilon}{n-1-\varepsilon}\int_{M^{n}}|A|^{n+2}\eta^{2}\rho+\Big(\frac{n}{2\varepsilon}\cdot\frac{n-1-2\varepsilon}{n-1-\varepsilon}
+1+\frac{1}{\delta}\Big) \int_{M^{n}}|\nabla\eta|^{2}|A|^{n}\rho \\
\leq&\frac{\tau n}{2}\cdot\frac{n-1-2\varepsilon}{n-1-\varepsilon}\Big(\int_{M^{n}}(|A|^{2})^{\frac{n}{2}}\Big)^{\frac{2}{n}}\cdot
\Big(\int_{M^{n}}(|A|^{n}\eta^{2}\rho)^{\frac{n}{n-2}}\Big)^{\frac{n-2}{n}}\\
&+\mathrm{C}(n, \varepsilon)\int_{M^{n}}|\nabla\eta|^{2}|A|^{n}\rho,
\endaligned
\end{equation}
where $\mathrm{C}(n, \varepsilon)$ a positive constant.
If we assume that
$$\Big(\int_{M^{n}}|A|^{n}\Big)^{\frac{1}{n}}<\sqrt{\frac{2}{\tau n\kappa}},$$
then there is $0<\varepsilon_{0}<1$  in \eqref{3.1-6} such that

\begin{equation*}
\aligned
\kappa^{-1}\Big(\int_{M^{n}}f^{\frac{2n}{n-2}}d\mu\Big)^{\frac{n-2}{n}}
\leq&\frac{\tau n}{2}\cdot\frac{n-1-2\varepsilon}{n-1-\varepsilon}\cdot\frac{2(1-\varepsilon_{0})}{\tau n\kappa}
\Big(\int_{M^{n}}(|A|^{n}\eta^{2}\rho)^{\frac{n}{n-2}}\Big)^{\frac{n-2}{n}}\\
&+\Big(\frac{n}{2\varepsilon}\cdot\frac{n-1-2\varepsilon}{n-1-\varepsilon}
+1+\frac{1}{\delta}\Big) \int_{M^{n}}|\nabla\eta|^{2}|A|^{n}\rho.
\endaligned
\end{equation*}
It can be seen that

\begin{equation*}
\aligned
\frac{(n-1-2\varepsilon)\varepsilon_{0}+\varepsilon}{\kappa(n-1-\varepsilon)}
\Big(\int_{M^{n}}f^{\frac{2n}{n-2}}d\mu\Big)^{\frac{n-2}{n}}
\leq\mathrm{C} (n, \varepsilon) \int_{M^{n}}|\nabla\eta|^{2}|A|^{n}\rho.
\endaligned
\end{equation*}
 The desired conclusion $|A|\equiv0$ then follows upon choosing an appropriate cut-off function $\eta$ since $\int_{M}|A|^{n}e^{\frac{|x|^{2}}{2}}d\mu<\infty$.

 \vskip10mm
\section{Proof of the rigidity theorem for $\mathring{A}$}

\vskip2mm
\noindent

The proof of the rigidity theorem in this section requires the following Sobolev inequality for submanifolds in the Euclidean
space (see \cite{HS},\cite{XG}).

\begin{proposition}\label{proposition 4.1}
Let $M^{n}$ ($n\geq3$) be an $n$-dimensional complete submanifold in the Euclidean space $\mathbb{R}^{n+p}$.
Let $f$ be a nonnegative $\mathrm{C}^{1}$
function with compact support. Then for all $s\in \mathbb{R}^{+}$,  we have

\begin{equation*}
\aligned
\|f\|^{2}_{2n/(n-2)}\leq D^{2}(n)\Big[\frac{4(n-1)^{2}(1+s)}{(n-2)^{2}}\|\nabla f\|^{2}_{2}+\big(1+\frac{1}{s}\big)\frac{1}{n^{2}}\||\vec{H}|f\|^{2}_{2}\Big],
\endaligned
\end{equation*}
where $D(n)=2^{n}(1+n)^{(n+1)/n}(n-1)^{-1}\sigma_{n}^{-\frac{1}{n}}$ and $\sigma_{n}$ denotes the volume of the unit ball in $\mathbb{R}^{n}$.
\end{proposition}

{\it Proof of Theorem \ref{theorem 1.2}}. Given a smooth compactly supported function $\eta$,
multiplying both sides of \eqref{2.1-4} by the factor $|\mathring{A}|^{2a-2}\eta^{2}\rho$ and following the same procedure as before, we obtain, for every $0<\epsilon<2a-1$,
\begin{equation*}
\aligned
&(2a-1-\epsilon)\int_{M^{n}}|\nabla |\mathring{A}||^{2}|\mathring{A}|^{2a-2}\eta^{2}\rho \\
\leq&
\int_{M^{n}}|\mathring{A}|^{2a}\eta^{2}\rho+\varsigma\int_{M^{n}}|\mathring{A}|^{2a+2}\eta^{2}\rho
+\frac{1}{n}\int_{M^{n}}|\vec{H}|^{2}|\mathring{A}|^{2a}\eta^{2}\rho+\frac{1}{\epsilon}\int_{M^{n}}|\nabla\eta|^{2}|\mathring{A}|^{2a}\rho.
\endaligned
\end{equation*}
Setting $f=|\mathring{A}|^{a}\eta\rho^{\frac{1}{2}}$. Using the same argument as in the proof of Theorem \ref{theorem 1.1}, we get
\begin{equation}\label{4.1-1}
\aligned
\int_{M^{n}}|\nabla f|^{2}
=&\int_{M^{n}}|\nabla (|\mathring{A}|^{a}\eta)|^{2}\rho+\int_{M^{n}}|\mathring{A}|^{2a}\eta^{2}|\nabla\rho^{\frac{1}{2}}|^{2}+\frac{1}{2}\int_{M^{n}}\langle \nabla \rho, \nabla(|\mathring{A}|^{2a}\eta^{2})\rangle \\
=&\int_{M^{n}}|\nabla (|\mathring{A}|^{a}\eta)|^{2}\rho+\frac{1}{4}\int_{M^{n}}|\mathring{A}|^{2a}\eta^{2}\rho|x^{T}|^{2}-\frac{1}{2}\int_{M^{n}}|\mathring{A}|^{2a}\eta^{2}\Delta\rho \\
=&\int_{M^{n}}|\nabla (|\mathring{A}|^{a}\eta)|^{2}\rho-\frac{1}{4}\int_{M^{n}}|\mathring{A}|^{2a}\eta^{2}\rho|x^{T}|^{2}-\frac{1}{2}\int_{M^{n}}|\vec{H}|^{2}|\mathring{A}|^{2a}\eta^{2}\rho \\
&-\frac{n}{2}\int_{M^{n}}|\mathring{A}|^{2a}\eta^{2}\rho.
\endaligned
\end{equation}
Combining the Cauchy inequality, \eqref{4.1-1} with the Sobolev inequality in Proposition \ref{proposition 4.1}, we obtain for any $\delta>0$

\begin{equation*}
\aligned
&\Big(\int_{M^{n}}f^{\frac{2n}{n-2}}\Big)^{\frac{n-2}{n}}\\
\leq&\frac{4 D^{2}(n)(n-1)^{2}(1+s)}{(n-2)^{2}}\int_{M^{n}}|\nabla f|^{2}+D^{2}(n)\big(1+\frac{1}{s}\big)\frac{1}{n^{2}}\int_{M^{n}}|\vec{H}|^{2}|f|^{2} \\
\leq&\frac{4 D^{2}(n)(n-1)^{2}(1+s)}{(n-2)^{2}}\Big(\int_{M^{n}}|\nabla (|\mathring{A}|^{a}\eta)|^{2}\rho-\frac{1}{2}\int_{M^{n}}|\vec{H}|^{2}|\mathring{A}|^{2a}\eta^{2}\rho-\frac{n}{2}\int_{M^{n}}|\mathring{A}|^{2a}\eta^{2}\rho\Big)\\
&+D^{2}(n)\big(1+\frac{1}{s}\big)\frac{1}{n^{2}}\int_{M^{n}}|\vec{H}|^{2}|\mathring{A}|^{2a}\eta^{2}\rho \\
\leq&\frac{4 D^{2}(n)(n-1)^{2}(1+s)}{(n-2)^{2}}\Big(a^{2}(1+\delta)\int_{M^{n}}|\nabla |\mathring{A}||^{2}|\mathring{A}|^{2a-2}\eta^{2}\rho+(1+\frac{1}{\delta})\int_{M^{n}}|\nabla \eta|^{2}|\mathring{A}|^{2a}\rho\\
&-\frac{1}{2}\int_{M^{n}}|\vec{H}|^{2}|\mathring{A}|^{2a}\eta^{2}\rho-\frac{n}{2}\int_{M^{n}}|\mathring{A}|^{2a}\eta^{2}\rho\Big)
+D^{2}(n)\big(1+\frac{1}{s}\big)\frac{1}{n^{2}}\int_{M^{n}}|\vec{H}|^{2}|\mathring{A}|^{2a}\eta^{2}\rho \\
\leq&\frac{4D^{2}(n)(n-1)^{2}(1+s)}{(n-2)^{2}}\Bigg(\frac{a^{2}(1+\delta)}{2a-1-\epsilon}\Big(\int_{M^{n}}|\mathring{A}|^{2a}\eta^{2}\rho+\varsigma\int_{M^{n}}|\mathring{A}|^{2a+2}\eta^{2}\rho\\
&+\frac{1}{n}\int_{M^{n}}|\vec{H}|^{2}|\mathring{A}|^{2a}\eta^{2}\rho+\frac{1}{\epsilon}\int_{M^{n}}|\nabla\eta|^{2}|\mathring{A}|^{2a}\rho\Big)+(1+\frac{1}{\delta})\int_{M^{n}}|\nabla \eta|^{2}|\mathring{A}|^{2a}\rho\\
&-\frac{1}{2}\int_{M^{n}}|\vec{H}|^{2}|\mathring{A}|^{2a}\eta^{2}\rho-\frac{n}{2}\int_{M^{n}}|\mathring{A}|^{2a}\eta^{2}\rho\Bigg)
+D^{2}(n)\big(1+\frac{1}{s}\big)\frac{1}{n^{2}}\int_{M^{n}}|\vec{H}|^{2}|\mathring{A}|^{2a}\eta^{2}\rho.
\endaligned
\end{equation*}
Choosing
\begin{equation}\label{4.1-2}
\aligned
&\kappa=\frac{4D^{2}(n)(n-1)^{2}}{(n-2)^{2}},\\
&\delta=\delta(\epsilon, s, a, n)=\frac{(2a-1-\epsilon)[2n^{2}(n-1)^{2}s-(n-2)^{2}]}{4a^{2}n(n-1)^{2}s}-1,
\endaligned
\end{equation}
where
$s$ and $\epsilon$
are positive constants satisfying
$$s>\frac{(2a-1-\epsilon)(n-2)^{2}}{2n(2an-n-2a^{2}-n\epsilon)(n-1)^{2}}, \ \ \epsilon\in(0, 2a-1-\frac{2a^{2}}{n})$$ such that $\delta>0$ for $1\leq a<\frac{n+\sqrt{n^{2}-2n}}{2}$.
Thus
\begin{equation}\label{4.1-3}
\aligned
&\kappa^{-1}\Big(\int_{M^{n}}f^{\frac{2n}{n-2}}\Big)^{\frac{n-2}{n}}\\
\leq&(1+s)\Bigg(\frac{a^{2}(1+\delta)}{2a-1-\epsilon}\Big(\int_{M^{n}}|\mathring{A}|^{2a}\eta^{2}\rho+\varsigma\int_{M^{n}}|\mathring{A}|^{2a+2}\eta^{2}\rho+\frac{1}{\epsilon}\int_{M^{n}}|\nabla\eta|^{2}|\mathring{A}|^{2a}\rho\Big)\\
&+(1+\frac{1}{\delta})\int_{M^{n}}|\nabla \eta|^{2}|\mathring{A}|^{2a}\rho-\frac{n}{2}\int_{M^{n}}|\mathring{A}|^{2a}\eta^{2}\rho\Bigg) \\
\leq&(1+s)\Bigg(\frac{\varsigma a^{2}(1+\delta)}{2a-1-\epsilon}\int_{M^{n}}|\mathring{A}|^{2a+2}\eta^{2}\rho
+\Big(\frac{a^{2}(1+\delta)}{\epsilon(2a-1-\epsilon)}+1+\frac{1}{\delta}\Big)\int_{M^{n}}|\nabla \eta|^{2}|\mathring{A}|^{2a}\rho\Bigg) \\
\endaligned
\end{equation}
By the H\"{o}lder inequality, it follows from \eqref{4.1-2} and \eqref{4.1-3} that
\begin{equation}\label{4.1-4}
\aligned
&\kappa^{-1}\Big(\int_{M^{n}}f^{\frac{2n}{n-2}}\Big)^{\frac{n-2}{n}} \\
\leq&\frac{\varsigma(1+s)[2n^{2}(n-1)^{2}s-(n-2)^{2}] }{4n(n-1)^{2}s}\Big(\int_{M^{n}}|\mathring{A}|^{n}\Big)^{\frac{2}{n}}
\Big(\int_{M^{n}}(|\mathring{A}|^{2a}\eta^{2}\rho)^{\frac{n}{n-2}}\Big)^{\frac{n-2}{n}} \\
&+\mathrm{C}(\epsilon, s, a, n)\int_{M^{n}}|\nabla \eta|^{2}|\mathring{A}|^{2a}\rho,
\endaligned
\end{equation}
where $\mathrm{C}(\epsilon, s, a, n)$ a positive constant.

\noindent
Since
$s>\frac{(2a-1-\epsilon)(n-2)^{2}}{2n(2an-n-2a^{2}-n\epsilon)(n-1)^{2}}$ and $\epsilon\in(0, 2a-1-\frac{2a^{2}}{n})$, for simplicity, we choose
\begin{equation*}
s(\epsilon, a)=\frac{(2a-1)(n-2)^{2}}{2n(2an-n-2a^{2}-n\epsilon)(n-1)^{2}}.
\end{equation*}
Let
\begin{equation*}
\Omega(s, n)=\frac{\kappa\varsigma(1+s)[2n^{2}(n-1)^{2}s-(n-2)^{2}] }{4n(n-1)^{2}s}, \ \ \Omega(a, n)=\inf_{\epsilon\in(0, 2a-1-\frac{2a^{2}}{n})}\Omega(s(\epsilon, a), n),
\end{equation*}
we know
\begin{equation*}
\aligned
\Omega(a, n)&=\Omega(s(0, a), n)\\
&=\frac{\varsigma D^{2}(n)(1+s(0, a))[2n^{2}(n-1)^{2}s(0, a)-(n-2)^{2}]}{n(n-2)^{2}s(0, a)} \\
&=\frac{4a^{2}\varsigma D^{2}(n)[\frac{(2a-1)(n-2)^{2}}{2n(2an-n-2a^{2})}+(n-1)^{2}]}{(2a-1)(n-2)^{2}}.
\endaligned
\end{equation*}
Under the assumption that
$$\Big(\int_{M^{n}}|\mathring{A}|^{n}\Big)^{\frac{1}{n}}<\sqrt{\frac{(2a-1)(n-2)^{2}}{4a^{2}\varsigma D^{2}(n)[\frac{(2a-1)(n-2)^{2}}{2n(2an-n-2a^{2})}+(n-1)^{2}]}},$$
there exists a positive constant $\bar{\Omega}>\Omega(a, n)$ such that
$$\Big(\int_{M^{n}}|\mathring{A}|^{n}\Big)^{\frac{1}{n}}<\bar{\Omega}^{-\frac{1}{2}}.$$
Consequently, one can select $\epsilon_{0}\in(0, 2a-1-\frac{2a^{2}}{n})$ satisfying

$$\Omega(a, n)<\Omega(s(\epsilon_{0}, a), n)<\bar{\Omega}.$$
It then follows that there exists $0<\varepsilon_{0}<1$  in \eqref{4.1-4} such that
\begin{equation*}
\aligned
\kappa^{-1}\Big(\int_{M^{n}}f^{\frac{2n}{n-2}}\Big)^{\frac{n-2}{n}}
\leq&\kappa^{-1}\cdot\Omega(s(\epsilon_{0}, a), n)\cdot\bar{\Omega}^{-1}\Big(\int_{M^{n}}(|\mathring{A}|^{2a}\eta^{2}\rho)^{\frac{n}{n-2}}\Big)^{\frac{n-2}{n}} \\
&+C(\epsilon_{0}, s, a, n)\int_{M^{n}}|\nabla \eta|^{2}|\mathring{A}|^{2a}\rho \\
\leq&\kappa^{-1}(1-\varepsilon_{0})\Big(\int_{M^{n}}(|\mathring{A}|^{2a}\eta^{2}\rho)^{\frac{n}{n-2}}\Big)^{\frac{n-2}{n}} \\
&+C(\epsilon_{0}, s, a, n)\int_{M^{n}}|\nabla \eta|^{2}|\mathring{A}|^{2a}\rho.
\endaligned
\end{equation*}
Namely,
\begin{equation*}
\aligned
\kappa^{-1}\varepsilon_{0}\Big(\int_{M^{n}}(|\mathring{A}|^{2a}\eta^{2}\rho)^{\frac{n}{n-2}}\Big)^{\frac{n-2}{n}}
\leq C(\epsilon_{0}, s, a, n)\int_{M^{n}}|\nabla \eta|^{2}|\mathring{A}|^{2a}\rho \\
\endaligned
\end{equation*}
When $\int_{M}|\mathring{A}|^{2a}e^{\frac{|x|^{2}}{2}}d\mu<\infty$, by an appropriate choice of the cut-off function $\eta$, we infer
$$|\mathring{A}|\equiv0, \ \ A=\frac{1}{n}g\otimes \vec{H}.$$ From $A=\frac{1}{n}g\otimes \vec{H}$ and the Codazzi equation,
we know
$$g_{ij}\nabla^{N}_{k}\vec{H}=g_{kj}\nabla^{N}_{i}\vec{H}.$$
Contracting over $i,j$, we obtain
$$(n-1)\nabla^{N}_{k}\vec{H}=0,  \ \ 1\leq k\leq n.$$
So that $|\vec{H}|^{2}$ is constant.  Substituting into \eqref{2.1-1}, we obtain
$$|\vec{H}|^{2}=0, \ \ |A|^{2}=0,$$
which implies that $M^{n}$ is isometric to $\mathbb R^{n}$.

\noindent{\bf Acknowledgements}
The first author was partially supported by NSFC grant No.12401060, Natural Science Foundation of Henan Province grant No.262300421232.
The second author was partly supported by NSFC grant No.12171164.

\noindent{\bf Declarations}

\noindent{\bf Conflict of interest} There are no conflicts of interest with third parties.

\noindent{\bf Data availability} Data sharing not applicable to this article as no datasets were generated or analysed during the current study.

\end{document}